\documentclass[11pt]{article}
\usepackage[a4paper]{geometry}
\usepackage{amsfonts,bm,doi,enumitem,fancyhdr,graphicx,amssymb,amsthm,amsmath,mathrsfs,mleftright,sectsty,subcaption,tikz,tikz-3dplot,titlesec,wrapfig,xcolor}

\usepackage{biblatex} 
\usepackage[markup=underlined,todonotes={textsize=footnotesize}
]{changes}

\usepackage[export]{adjustbox}

\setanonymousname{Oleg}
\definechangesauthor[color=red]{James}

\newtheorem{theorem}{Theorem}[section]

\newtheorem{proposition}[theorem]{Proposition}
\newtheorem{corollary}[theorem]{Corollary}

\theoremstyle{remark}
\newtheorem{remark}[theorem]{Remark}
\theoremstyle{definition}
\newtheorem{definition}[theorem]{Definition}
\newtheorem{example}[theorem]{Example}
\newtheorem{problem}{Problem}

\def \r{\mathbb R}\def \R{\r}

\def \z{\mathbb Z}\def \Z{\z}

\DeclareMathOperator \aee{\mbox{{\rm \ae}}}

\DeclareMathOperator{\isin}{lsin}

\DeclareMathOperator{\itan}{ltan}
\DeclareMathOperator{\iarctan}{larctan}

\DeclareMathOperator{\sgn}{sgn}\DeclareMathOperator{\MOD}{mod}
\DeclareMathOperator{\il}{l\ell} \DeclareMathOperator{\is}{lS}
 
\DeclareMathOperator{\aff}{Aff}

\DeclareMathOperator{\GL}{GL} 

\DeclareMathOperator{\lls}{LLS}

\DeclareMathOperator{\arclls}{\mathcal{B}}

\def \({\langle}
\def \){\rangle}

\def \concat{\circ}
\def \emptyseq{\epsilon}

\title{Lattice angles of lattice polygons}

\author{James Dolan, Oleg Karpenkov}

\begin{document}

\maketitle

\begin{abstract}
This paper is dedicated to a lattice analog to the classical ``sum of interior angles of a polygon theorem''.
In 2008, the first formula expressing conditions on the geometric continued fractions for lattice angles of triangles was derived, while the cases of $n$-gons for $n > 3$ remained unresolved.
In this paper, we provide the complete solution for all integer $n$.
The main results are based on recent advances in geometry of continued fractions.
\end{abstract}

\section*{Introduction}

The {\it sum of interior angles of a polygon theorem} in Euclidean geometry states that
the sum of all interior angles of an $n$-gon equals $(n-2)\pi$; for any sequence of $n$ non-zero angles smaller than $\pi$, altogether summing to $(n-2)\pi$, there exists a convex $n$-gon with these angles.
In this paper, we study the lattice analog of this theorem.
Namely, we provide a solution to the following problem.
\begin{problem}\label{ikea-problem}
Describe all $n$-tuples of angles in lattice convex $n$-gons.
\end{problem}

\noindent(This problem is known in Folklore as IKEA problem: {\it how to fit a given collection of furniture items to an $n$-gonal room?})

\vspace{2mm}

We develop criteria for given lattice angles to be the angles of some convex lattice polygons (Theorem~\ref{IKEA-solution} and Theorem~\ref{IKEA-2}).
Without loss of generality, we restrict ourselves to the case of the integer lattice whose group of symmetries is the group of affine integer lattice preserving transformations $\aff(2,\Z)$.
This group is a semidirect product of the group of multiplication by $\GL(2,\Z)$ matrices and the group of translations on vectors with integer coordinates.
As a by-product, we obtain a new integer generalisation of the angle-side-angle rule for triangle congruence (Proposition~\ref{asca-rule}).

\vspace{2mm}

The case of triangles was answered in 2008 in~\cite{Karpenkov2008}.
The proof was given in terms of values of long continued fractions geometrically defined by integer angles
(for further details on integer geometry we refer to~\cite{Arnold2002,karpenkov2022geometry,Karpenkov2011}).
These fraction are closely related to the theory of integer trigonometric functions developed in~\cite{Karpenkov2008, Karpenkov2009a}. In this paper, we cover all of the remaining cases of $n$-gons (i.e. $n>3$), which completes the solution of Problem~\ref{ikea-problem}.

\vspace{2mm}

There are numerous publications related to various classifications of convex lattice polygons and their frequencies.
In particular, the statistics of integral convex polygons were studied in~\cite{Arnold1980} by V.~Arnold
and in~\cite{Barany1992}
by I.~B\'ar\'any and A.M.~Vershik.
In~\cite{Schicho2003,Castryck2015} (and later developed in~\cite{Harrison2022}) the authors study the notion of lattice sizes of lattice polygons. 
Further in~\cite{Henk2022} M.~Henk and S.~Kuhlmann studied lattice width of lattice-free polyhedra.
In~\cite{Morrison2021} R.~Morrison and A.K.~Tewari classified convex lattice polygons with all lattice point visible.
Gauss-Kuzmin statistics of such polygons were studied by M.~Kontsievoch~\cite{Kontsevich1999} and further related to hyperbolic geometry by the second author in~\cite{Karpenkov2007}.

\vspace{2mm}

The relation on the angles in convex polyhedra defines the global relations to cusp singularities in toric varieties~\cite{Tsuchihashi1983, Karpenkov2008} and their interplay with rigidity theory~\cite{Mohammadi2023}.
Collection of angles in the polygons impact the values of coefficients in Earhart polynomials, see e.g. in~\cite{Beck2007}.

\vspace{2mm}

\noindent
{\bf This paper is organised as follows.}
In Section~\ref{Preliminary definitions and formulation of the main statements}, we give all necessary definitions on continued fractions, sails of angles, and chord curvatures of broken lines.
We also adapt the generalised angle-side-angle rule for lattice geometry.
We formulate the main results on the relations between the angles of lattice $n$-gons in Section~\ref{Classification of the angles for convex lattice polygons}.
Further in Section~\ref{Further tools needed for the proof},
we introduce additional notions and definitions that are required for the proofs and prove some auxiliary statements.
Finally in Section~\ref{Proofs of the main results}, we complete the proofs of main results. 
We conclude this paper in Section~\ref{A few words on the multidimensional case} with a short discussion on the IKEA problem in the multidimensional case, which is still open.

\section{Preliminary definitions and formulation of the main sta\-tements}
\label{Preliminary definitions and formulation of the main statements}

In this section we give basic definitions of lattice geometry, continuants and continued fractions. 
We discuss locally convex broken lines and introduce the angle-curvatures sequences classifying them. Finally we present a new integer angle-side-curvature-angle rule.  

\subsection{The monoid of finite integer sequences under concatenation}

\vspace{2mm}

Denote the set of all integer sequences by
$$
\Z ^*
=\bigcup\limits_{n=0}^{\infty}
\Z^n,
$$
where $\Z^0=\{\emptyseq\}$ represents the empty sequence 
$\emptyseq=()$.
There is the following natural operation of concatenation 
acting on $\Z^*$.
Consider 
$$
X=(x_1,\dots,x_n)\in \Z^* \quad \hbox{and} 
\quad  Y=(y_1,\dots,y_m) \in \Z^*.
$$
Then the {\it concatenation} $X\concat Y$ is the sequence
$$
X\concat Y=(x_1,\dots,x_n,y_1,\dots,y_m).
$$

We have that $(\Z^*,\circ)$ is a monoid.

\subsection{Continuant and continued fractions}

Let us recall the classic notion of continued fractions.

\begin{definition}
Let $\alpha\in\R\cup\{\infty\}$ and $a_0,\ldots,a_n$ be real numbers satisfying the following equation
$$
\alpha=a_0+\cfrac{1}{ a_1+\cfrac{1}{ a_2+\cfrac{1}{\ddots+\cfrac{1}{a_n}}}}
$$
Then we say that the expression in the right-hand side of the equality is a {\it continued fraction expansion} of $\alpha$; we denoted this by $[a_0;a_1:\cdots :a_n]$.
We say that $n+1$ is the length of the continued fraction and that $a_0, \ldots, a_n$ are the {\it elements} of this continued fraction.
\end{definition}

\begin{remark}
While evaluating these expressions on the right hand side,
we might get $0$ intermediate values. 
Here are the following simple three rules to deal with them
$$
\frac10=\infty, \qquad
a+\infty=\infty, \qquad 
\frac{1}{\infty}=0.
$$
\end{remark}
The following function is closely related to the notion of the continued fraction.

\begin{definition}
The \textit{continuant function} $K:\Z^*\to\Z$ is defined iteratively as follows:
$$
\begin{array}{l}
K(\emptyseq)=1,\qquad
K(x_1) = x_1,\\
K(x_1,\dots,x_n)=x_nK(x_1,\dots,x_{n-1})+K(x_1,\dots,x_{n-2}),
\quad \hbox{for $n\ge 2$}.
\end{array}
$$
\end{definition}

\begin{remark}
Continuants and continued fractions are linked by the following formula.
$$
[a_0;a_1:\cdots:a_n]=
\frac{K(a_0,a_1,\ldots,a_n)}{K(a_1,\ldots,a_n)}.
$$
For more information on theory of continued fractions see~\cite{karpenkov2022geometry,Khinchin1961}.
\end{remark}

\subsection{Definitions of integer geometry}

We say that a point/vector in $\r^2$ is {\it integer} if its coordinates are integer. We say that an angle is {\it integer} if its vertex is integer. Further, an integer  angle is {\it rational} if both of its edges contain integer points distinct to the vertex. 
Finally a line is {\it integer} if it contains an least two distinct integer points.

\vspace{2mm}

Let us collect the following invariants of lattice geometry. 
\begin{samepage}
    
\begin{definition}\quad
\begin{itemize}
\item The {\it integer length} of an integer segment $AB$ is the number of connected components in $AB\setminus \Z^2$. Denote it by $\il (AB)$.

\vspace{2mm}

\item Let $AB$ and $AC$ be two integer segments,
then the value $\det(AB,AC)$ is invariant under 
the action of $\aff (2,\z)$.
It is called {\it oriented integer area} of the triangle $\triangle BAC$. 
Here the absolute value of  $\det(AB,AC)$ is called the {\it integer area} of $\triangle BAC$, it is denoted by $\is\triangle BAC$.

\item Let $AB$ and $AC$ be two integer segments,
then the value 
$$
\isin \angle BAC =\frac{\is \triangle ABC}{\il (AB)\il(AC)}
$$
is the {\it integer sine} of the angle $\angle BAC$.

\item Let $\ell$ be an integer line and $A$ be some integer point outside $\ell$. The {\it integer distance} between
$A$ and $\ell$ is the index of a sublattice generated all integer vectors starting at $A$ and ending at integer points of $\ell$ in the lattice $\Z^2$. 

\item Let $\ell_1$ and $\ell_2$ be two distinct parallel  integer lines. The {\it integer distance} between then
is the integer distance between any integer point of $\ell_1$ and the line $\ell_2$.

\end{itemize}

\end{definition}
\end{samepage}

\vspace{2mm} 

Recall that two sets $U$ and $V$ are integer congruent if there is an integer affine transformation (i.e., an affine integer lattice preserving transformation) that is a bijection between $U$ and $V$. We write $U\cong V$.

\subsection{Sails for angles and their LLS sequences, integer arctangents}

In this subsection we discuss the classical geometry of continued fractions. We start with the following definition.

\begin{definition}
Consider the convex hull of all integer points inside a rational angle $\alpha$ except for its vertex. Then the boundary of this convex hull is a broken line. The {\it sail} of $\alpha$ is the union of all bounded edges of the broken line.
\end{definition}
\begin{figure}
$$
\includegraphics[height=7cm]{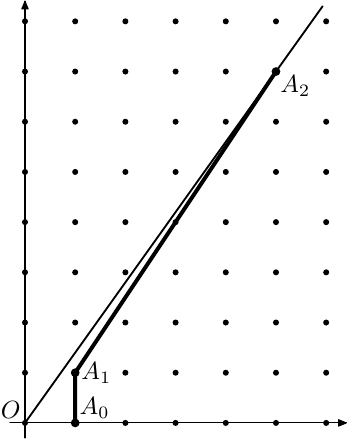}
$$
\caption{The sail for an angle.}
\label{sail-1}
\end{figure}

\begin{definition}
For a rational integer angle $\alpha$ with sail $A_0,\dots,A_{n+1}$, we define the \textit{LLS sequence} of $\alpha$ as the sequence $(a_0,\dots,a_{2n})$ where
$$
\begin{array}{rcl}
a_{2k}&=&\il(A_kA_{k+1});\\
a_{2k-1}&=&\isin(\angle A_{k-1}A_kA_{k+1}).
\end{array}
$$
We also write $\lls(\alpha)=(a_0,\dots,a_{2n})$. We then define the \textit{integer arctangent} of $\alpha$ as the continued fraction expansion $\itan(\alpha)=[a_0;a_1,\dots,a_{2n}]$.
\end{definition}

We show an example of the sail for the angle with vertices $(1,0)$, $(0,0)$, and $(5,7)$ on Figure~\ref{sail-1}.

\begin{remark}
Our approach is alternative to Hirzebruch-Jung continued fractions developed in~\cite{Hirzebruch1953} and~\cite{Jung1908}.
\end{remark}

\begin{remark}
Here LLS is an abbreviation of "lattice length sine". This comes from the fact that the LLS sequence of a sail is precisely the alternating sequence of integer lengths of its line segments and integer sines of the angles between them.
\end{remark}

Note that $\itan$ takes all rational values greater than or equal to $1$. 
In fact, integer tangent is a complete invariant of integer angles (\cite[Theorem~4.11]{karpenkov2022geometry}). There is a natural way to define the integer arctangent,
it is very similar to the Euclidean case.

\begin{definition}
Let $m\ge n$ be two  positive integers that are relatively prime.
The {\it integer arctangent} of $m/n$ is the following integer angle
$$
\iarctan\frac{m}{n}= \angle AOB_{m/n},
$$
where $A=(1,0)$, $O=(0,0)$, and $B_{m/n}=(n,m)$.
(See~\cite{karpenkov2022geometry} for further details.)
\end{definition}

\subsection{Locally convex broken lines and their chord curvature}

A {\it convex polygon} $P$ is the boundary of the convex hull of some finite set $S$ that spans the whole $\r^2$. A point of $P$ is a {\it vertex} if it does not belong to some open interval contained in the convex hull of $P$.
By this definition all the angles of convex polygons are smaller than $\pi$.

In this paper, we work with convex polygons that are the boundaries of the convex hull of finite sets of points not contained in one line. A convex $n$-gon $P$ is therefore given  by its vertices $A_1\dots A_n$ ordered anticlockwise from some arbitrary point $A_1$ and indexed via residues of $n$. In particular, we write $A_0=A_n$ and $A_{n+1}=A_1$.


\vspace{2mm}

We say that an angle $\angle ABC$ is {\it positively/negatively oriented} if the vectors $BA$ and $BC$ define a positive/negative basis in $\r^2$.
In case if $A$, $B$, and $C$ are in a line we say that the angle {\it does not have a natural orientation}.

\begin{definition}
    We say that an integer broken line is \textit{locally convex} if all its angles are simultaneously either positively oriented or negatively oriented.
\end{definition}

\begin{remark}
All the convex polygons are locally convex.
\end{remark}


\begin{definition}\label{chord-curvature-ABCD}
Consider a locally convex broken line $ABCD$ with no three consecutive points being in a line.
We say that the  \textit{chord curvature} of $ABCD$ is the following quantity
$$
\aee(ABCD)=
\il(BC) -
\sgn\langle BC, B'C'\rangle \cdot \il(B'C')  - 2,
$$
where 
\begin{itemize}
\item the point $B'$ is the integer point of the angle $ABC$ at the unit integer distance to the segment $BC$ which is the closest to the line $AB$;

\item the point $C'$ is the integer point of the angle $BCD$ at the unit integer distance to the segment $BC$ which is the closest to the line $CD$;

\item the value $\sgn\langle BC, B'C'\rangle$ is defined as follows:
$$
\sgn\langle BC, B'C'\rangle=
\left\{
\begin{array}{rl}
1, & \hbox{if $BC$ and $B'C'$ has the same direction;}\\
0, &  \hbox{if $B'=C'$;}\\
-1, & \hbox{if   $B'C'$ has the opposite direction to $BC$.}
\end{array}
\right.
$$
\end{itemize}
\end{definition}

\begin{example}
Let us consider the broken line $ABCD$ with points 
$$
A=(0,2), \qquad 
B=(4,0), \qquad 
C=(0,0), \quad \hbox{and} \quad 
D=(2,3) 
$$
(see Figure~\ref{figure-1.pdf}).
Then $B'=(1,1)$ and $C'=(2,1)$. 
Thus $\il (BC)=4$ and $\il (B'C')=1$
The vector $B'C'$ has the same direction with the vector $BC$,
and hence $\sgn\langle BC, B'C'\rangle=1$.
We have
$$
\aee(ABCD)=
4-1 \cdot 1 -2=1.
$$
\begin{figure}
$$
\includegraphics[width=5cm]{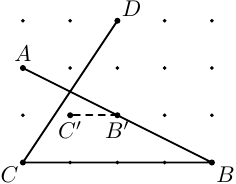}
$$
\caption{The broken line $ABCD$ and its points $B'$ and $C'$.}
\label{figure-1.pdf}
\end{figure}
\end{example}

Let us finally give the definition of chord curvature for broken lines on $n$ vertices.
\begin{definition}\label{chord-curvature}
Let $A_1\ldots A_n$ be a locally convex broken line. We say that the {\it chord curvature} at an edge $A_iA_{i+1}$ (where $2\le i \le n-2$) is the chord curvature of $A_{i-1}A_iA_{i+1}A_{i+2}$.  
\end{definition}

\subsection{Angle-curvature sequences for locally convex broken lines and convex polygons}

As we have already noticed above the chord curvatures are natural lattice invariants to describe locally convex broken lines and polygons.

\vspace{2mm}

Let us now define a lattice invariant that will playing the leading role in our main statements.
\begin{definition}
Let $A_0A_1\ldots A_nA_{n+1}$ be a locally convex broken line.
Set
$$
\begin{array}{l}
\alpha_i=\iarctan(\angle A_{i-1}A_iA_{i+1}), \quad \hbox{for $i=1,\ldots, n$};\\
\aee_i=\aee(A_{i-1}A_iA_{i+1}A_{i+2}), \quad \hbox{for $i=1,\ldots, n-1$}.\\
\end{array}
$$
We say that the sequence 
$$
\mathcal{S}=(\alpha_1,\aee_1,\ldots, \alpha_{n-1},\aee_{n-1}, \alpha_n)
$$ 
is the {\it angle-curvature sequence} for the broken line. 
\end{definition}

\begin{definition}
Let $A_1\ldots A_n$ be a locally convex polygon.
Set
$$
\begin{array}{l}
\alpha_i=\iarctan(\angle A_{i-1}A_iA_{i+1}), \quad \hbox{for $i=1,\ldots, n$};\\
\aee_i=\aee(A_{i-1}A_iA_{i+1}A_{i+2}), \quad \hbox{for $i=1,\ldots, n$}.\\
\end{array}
$$
We say that the sequence 
$$
\mathcal{S}=(\alpha_1,\aee_1,\ldots, \alpha_{n-1},\aee_{n-1}, \alpha_n, \aee_{n})
$$ 
is the {\it angle-curvature sequence} for the polygon. 
\end{definition}

\begin{remark}
In case if we do not fix a starting point of the polygon, the corresponding angle-curvature sequence is considered to by cyclic.
\end{remark}

Similar to triangles we give the following criterion of integer congruence of integer broken lines.

\begin{proposition}
Two broken locally convex broken lines $($or polygons$)$ are integer congruent if and only if their angle-curvature sequences coincide and if their sequences of integer lengths coincide.
\qed
\end{proposition}

\subsection{Discussion on integer angle-side-curvature-angle rule}

From Euclidean geometry we know the classical {\it ASA $($angle-side-angle$)$ rule}: If two pairs of angles of two triangles are equal, and the included sides are equal, then these triangles are integer congruent.

\vspace{2mm}

It turns out that this rule is not sufficient for lattice geometry. Let us illustrate this with the following example

\begin{example}
Let $A=(0,0)$, $B=(2,0)$, $C=(1,1)$, $C'=(0,2)$.
Consider triangles $\triangle ABC$ and $\triangle ABC'$.
Then
$$
\angle BAC\cong \angle BAC'\cong \angle ABC\cong \angle ABC' \cong\iarctan 1
$$
and both triangles share the same segment $ABC$.
However their lattice areas are 
$$
\is(ABC)=2 \quad \hbox{and} \is (ABC')=4,
$$
and hence these triangles are not integer congruent to each other.
\end{example}

The notion of chord curvature providing us a new refined rule.

\begin{proposition}\label{asca-rule}
Two triangles $\triangle ABC$ and $\triangle A'B'C'$
are integer congruent $($preserving the order of points$)$
if and only if the following 
{\it integer angle-side-curvature-angle rule} $($or {\it ASCA rule} for short$)$ holds:
$$
\begin{array}{ll}
\angle ABC\cong\angle A'B'C', \qquad &  \angle BAC\cong\angle B'A'C',\\
\il (AB)=\il(A'B'), \qquad &
\aee (CABC)= \aee (C'A'B'C').\\
\end{array}
$$
\end{proposition}

\begin{proof}
Let all the ASCA rule condition hold.
Consider the transformation sending the angle $\angle C'A'B'$ to the angle $\angle CAB$. Here the images of $B'$ and $C'$ are some points $B''$ and $C''$
Since $\il (AB)=\il(A'B')=\il (AB'')$, we have $B''=B$.

\vspace{2mm}

Since $\aee (C''AB''C'')=\aee (C'A'B'C')=\aee (CABC)$ the first edges of the sails of  $\angle ABC''$ and $\angle ABC$ have the same direction.
Since  $\angle BAC''\cong\angle B'A'C'\cong\angle BAC$, the segment $AB$ is a common edge for the triangles $\triangle ABC''$
and $\triangle ABC''$, and the sails for $\angle ABC''$ and $\angle ABC$ have the same direction, we get that the sails for
$\angle ABC''$ and $\angle ABC$ coincide. This implies that the line $BC$ coincides with the line $BC''$.

\vspace{2mm}

The lines $AC$ and $AC''$ coincide by the construction 
and the lines $BC$ coincides with the line $BC''$ from the above. 
Hence $C=C''$. Therefore, the triangles  $\triangle ABC$ and $\triangle A'B'C'$ are integer congruent.

\vspace{2mm}

Conversely, if the triangles  $\triangle ABC$ and $\triangle A'B'C'$ are integer congruent, then all the listed identities in the ASCA rule hold, since the corresponding quantities are lattice invariants.
\end{proof}

\section{Classification of the angles for convex lattice polygons}
\label{Classification of the angles for convex lattice polygons}
In this subsection, we formulate the main results of this paper.
The proofs are provided later in Section~\ref{Proofs of the main results}.

\subsection{Solution to IKEA problem}

\vspace{2mm}

We start with the following notation.

\begin{definition}\label{mathcal-S-n}
For an angle-curvature sequence 
$$
\mathcal{S}=(\alpha_1,\aee_1,\ldots, \alpha_{n-1},\aee_{n-1}, \alpha_n, \aee_n)
$$ 
we set $\mathcal S_j^k$ to be the subsequence $(\alpha_j,\aee_j,\dots,\alpha_{k-1}, \aee_{k-1},\alpha_k)$. 
We then write 
$$
\lls(\mathcal S_j^k)=\lls(\alpha_j)\concat (\aee_j)\concat \cdots \concat \lls(\alpha_{k-1})\concat (\aee_{k-1})\concat \lls(\alpha_{k}).
$$
\end{definition}

\begin{definition}
    We say an integer sequence {\it changes sign $n$ times} if after removing all zero elements it contains exactly $n$ pairs of consequent elements with opposite signs.
\end{definition}

We are ready to formulate the first main result of the paper.

\begin{theorem}\label{IKEA-solution}
Consider the set of integer angles $\alpha_1, \ldots, \alpha_n$
and the integers $\aee_1,\ldots, \aee_{n}$.
Then there exists a convex $n$-gon 
with the prescribed angle-curvature sequence,
$$
\mathcal{S}=(\alpha_1,\aee_1,\ldots, \alpha_{n-1}\aee_{n-1}, \alpha_n, \aee_n),
$$ 
if and only if the following conditions hold:
\begin{itemize}
\item 
$K(\lls(\mathcal S_1^n))=0$;
\item $
\displaystyle \aee_{n} = -\left\lfloor \frac{K(\lls(\mathcal S_2^n)\concat (1))}{K(\lls(\mathcal S_2^n))} \right\rfloor$;
\item 
The sequence
$(K(\lls(\mathcal S_1^j)))_{j=1}^{n}$
changes the sign 
$n-3$ times.
\end{itemize}
\end{theorem}

\begin{remark}
Once the last condition is removed, we have a similar statement for locally convex closed broken lines. 
\end{remark}

\begin{remark}
Practically, one may use continued fractions in order to compute the values of continuants. Informally speaking, the continuant of a sequence coincides with the numerator of the continued fraction for the same sequence.
An important rule here is to keep the denominators (without cancelling their factors  with the factors of the numerators). Let us illustrate this with a sample computation of $K(-1,2,-3)$. We have:
$$
[-1;2:-3]=
-1+\frac{1}{\displaystyle 2+\frac{1}{\displaystyle
-3}}
=
-1+\frac{1}{\displaystyle \frac{-6}{-3}+\frac{1}{\displaystyle
-3}}
=
-1+\frac{-3}{-5}
=\frac{5}{-5}+\frac{-3}{-5}
=+\frac{2}{-5}.
$$
Hence $K(-1,2,-3)=2$.
\end{remark}

Let us supplement Theorem~\ref{IKEA-solution}
with the following example.

\begin{example}\label{example-IKEA-1}
Consider a quadrangle  with vertices $A=(4,-1)$, $B=(0,0)$,
$C=(2,3)$, and $D=(3,3)$ as on Figure~\ref{figure-4.pdf}.
\begin{figure}
$$
\includegraphics[width=5cm]{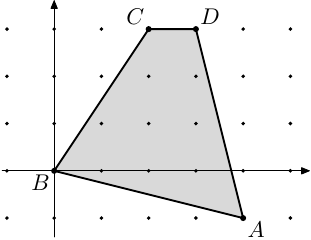}
$$
\caption{The quadrangle $ABCD$.}
\label{figure-4.pdf}
\end{figure}
Direct computations show that the angle-curvature sequence for this quadrangle is as follows:
$$
\begin{array}{r}
\mathcal{S}=\big(\iarctan([1;3:1:1:1]), -1,
\iarctan([3]), -2,
\iarctan([1;2,1]), -1,   \qquad\quad\\
\iarctan([3;1:3]),-1 
\big). 
\end{array}
$$
Let us check the first condition of Theorem~\ref{IKEA-solution}.
Indeed we have
$$
K(\lls(\mathcal S_1^4))=K(1,3,1,1,1,-1,3,-2,1,2,1,-1,3,1,3)=0.
$$
The second condition is also true:
$$
\aee_4=-1=-\frac{17}{14}=
-\left\lfloor
\frac{K(3,-2,1,2,1,-1,3,1,3,1)}{K(3,-2,1,2,1,-1,3,1,3)}
\right\rfloor
=
\left\lfloor
\frac{K(\lls(\mathcal S_2^4)\concat (1))}{K(\lls(\mathcal S_2^4))}
\right\rfloor
.
$$
For the last condition the corresponding sequence of continuants is as follows:
$$
\begin{array}{l}
K(\lls(\mathcal S_1^1))=K(1,3,1,1,1)=14;\\
K(\lls(\mathcal S_1^2))=K(1,3,1,1,1,-1,3)=-1;\\
K(\lls(\mathcal S_1^3))=K(1,3,1,1,1,-1,3,-2,1,2,1)=-15;\\
K(\lls(\mathcal S_1^4))=K(1,3,1,1,1,-1,3,-2,1,2,1,-1,3,1,3)=0.\\
\end{array}
$$
We have a sequence that changes the sign once, exactly $n-3=4-3=1$ times, and thus the last condition also holds.
\end{example}

\subsection{Explicit construction of the last entry in the angle-curvature sequence}

Let us fix some notation and definitions required for the next theorem.
Let $S$ be some finite sequence. We write $S^t$ for the reversed sequence, and we write $-S$ for the sequence with negated terms. 

\begin{definition}\label{angle-sequence}
Let $(a_0,\ldots, a_n)$ 
be an arbitrary sequence of integers.
Set
$$
\angle (a_0,\ldots, a_n) =\angle AOB,
$$
where $A=(1,0)$, $O=(0,0)$, and 
$C=\big(K(a_1,\ldots,a_{2n}),K(a_0,a_1,\ldots,a_{2n})\big)$.
\end{definition}

Below is the second main result of the paper.

\begin{theorem}\label{IKEA-2}
    Given arbitrary non-zero angles $\alpha_1,\dots,\alpha_{n}$ and integers $\aee_1,\ldots, \aee_{n-1}$ where $n\geq2$, there exists a unique locally convex $(n+1)$-gon whose angle-curvature sequence is 
    $$
    \mathcal{S}=(\alpha_1,\aee_1,\ldots, ,\alpha_{n-1}\aee_{n-1}, \alpha_n, x, \beta, y)
    $$
    for some integers $x$ and $y$ and an integer angle $\beta$.
    \\
    The values of $x$, $y$,  and $\beta$ are defined as follows:

    $$
    x = -\left\lfloor \frac{K(U\concat (1))}{K(U)} \right\rfloor
    , \qquad
    \beta \cong \angle\big(-U^t\big)
    ,\qquad
    y = -\left\lfloor \frac{K(V\concat (1))}{K(V)} \right\rfloor,
    $$
where
    $$
\begin{array}{l}
    U=\lls(\mathcal{S}_1^n)=\lls(\alpha_1)\concat (\aee_1)\concat \cdots \concat \lls(\alpha_{n-1})\concat(\aee_{n-1})\concat \lls(\alpha_{n});
\\ 
V=\lls(\mathcal{S}_2^{n+1})=\lls(\alpha_2)\concat (\aee_2)\concat \cdots \concat \lls(\alpha_{n-1})\concat (\aee_{n-1})
\concat \lls(\alpha_{n})\\
\qquad\qquad\quad\qquad\qquad\qquad\qquad
\qquad\qquad\qquad\qquad\qquad\qquad\concat (x)\concat \lls(\beta).
\end{array}
    $$
\end{theorem}

\begin{remark}
Note that
$$
y=\left\lfloor \frac{K(-U^t\concat (-1))}{K(-U^t)} \right\rfloor
$$
in the notation of the above theorem.
\end{remark}

\begin{example}
Consider three angles $\alpha_1, \alpha_2$ and 
$\alpha_3$ with LLS sequences $(1,3,1,1,1)$, $(3)$, and $(1,2,1)$.
Let also $\aee_1=-1$ and $\aee_2=-2$.
Then
  \begin{align*}
    U&=\lls(\alpha_1)\concat (\aee_1)\concat \lls(\alpha_2)\concat (\aee_{2})\concat \lls(\alpha_{3})\\
    &=(1,3,1,1,1) \concat(-1)\concat(3)\concat(-2)\concat(1,2,1)\\
    &=(1,3,1,1,1,-1,3,-2,1,2,1)
  \end{align*}
and hence
$$
 x = -\left\lfloor \frac{K(U\concat (1))}{K(U)} \right\rfloor
=
-\left\lfloor \frac{K(1,3,1,1,1,-1,3,-2,1,2,1,1)}{K(1,3,1,1,1,-1,3,-2,1,2,1)} \right\rfloor
=
-\left\lfloor \frac{-26}{-15} \right\rfloor
=-1.
$$
Further we have
$$
-U^t= (-1, -2, -1, 2, -3, 1, -1, -1, -1, -3, -1).
$$
Now we compute the continued fraction corresponding to the last sequence:
$$
[-1; -2: -1: 2: -3: 1: -1: -1: -1: -3: -1]=\frac{15}{-11}
$$
(here we take care of the signs for continuants in the fraction).
Then $\beta$ is integer congruent to the angle with vertices
$$
(1,0), \quad (0,0), \quad \hbox{and} \quad (-11, 15
).
$$
In terms of arctangents we have
$$
\beta \cong \iarctan\Big(\frac{15}{-11 (\MOD 15)}\Big)
=
\iarctan\Big(\frac{15}{4}\Big).
$$
Since $15/4=[3;1;3]$, we get $\lls(\beta)=(3,1,3)$.
Finally,
$$
    V=\lls(\alpha_2)\concat(\aee_2)\concat \lls(\alpha_3)\concat (x)\concat \lls(\beta)=(3,-2,1,2,1,-1,3,1,3),
    $$
and hence
$$
 y = -\left\lfloor \frac{K(V\concat (1))}{K(V)} \right\rfloor
=
-\left\lfloor \frac{K(3,-2,1,2,1,-1,3,1,3,1)}{K(3,-2,1,2,1,-1,3,1,3)} \right\rfloor
=
-\left\lfloor \frac{-17}{-14} \right\rfloor
=-1.
$$
(This example correspond to the quadrangle of Example~\ref{example-IKEA-1}, the quadrangle is shown on
Figure~\ref{figure-4.pdf} above.)
\end{example}

\section{Further tools needed for the proof}
\label{Further tools needed for the proof}

This subsection contains some extra definitions, notions, and 
statements that are used in the proofs of the next section.

\subsection{Sail diagrams for broken lines}

Let us generalise the notion of the sail for an integer angle to the case of broken lines with integer vertices.

\begin{definition}
The {\it sail diagram} of a broken line (closed or not closed) with integer vertices is the broken line obtain by the sails of the  consecutive angles of the broken line shifted to the origin (in composition with the central symmetry for all even ones). 
\\
The vertices of the sail diagram corresponding to the edges of the angles are called {\it edge vertices}.
\end{definition}

\begin{remark}
The sail diagram for a broken line coincides with the sail diagram of the sum of its interior angles corresponding to integer parallel transform of these angles to the origin and symmetries about the origin for angles with even numbers.     
\end{remark}

\begin{remark}
In case of a convex polygon with even number of vertices its sail diagram is a closed broken line. If the number of vertices is odd, then the starting and the end points of the sail diagram are two points symmetric with respect to the origin.
\end{remark}

\begin{remark}
    Sail diagrams have a flavour of Maxwell reciprocal diagrams in rigidity theory.
\end{remark}

\begin{example}
Let us consider a pentagon (see on Figure~\ref{figure-2-3.pdf}, Left) with vertices 
$$
(8,0), \quad
(0,0), \quad
(2,3), \quad
(3,4), \quad \hbox{and} \quad
(5,3).
$$
Then its angle-curvature sequence is
$$
(\iarctan(3/2), -2, \iarctan(1), -4, \iarctan(3), -2, \iarctan(1), -3, \iarctan(1), 0).
$$
The corresponding sail diagram (see on Figure~\ref{figure-2-3.pdf}, Right) has vertices
$$
(1,0), \quad 
(1,1), \quad  
(2,3), \quad 
(-1,-1), \quad 
(2,-1), \quad 
(-1,1), \quad \hbox{and} \quad
(-1,0). 
$$
All these vertices except $(1,1)$ are edge vertices.

\begin{figure}
    \centering
    \includegraphics[height=5.5cm]{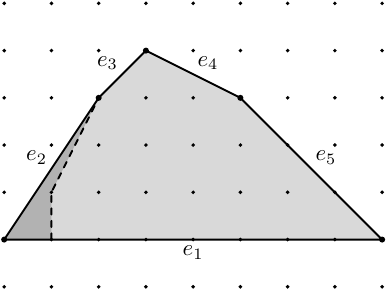}
    \hspace{2em}
    \includegraphics[ bb = 0 4 143 143, height=5.5cm]{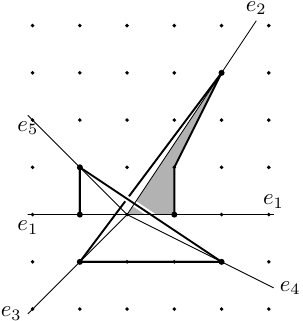}
    \caption{A pentagon with marked (on the left) edges and its sail diagram (on the right).
We indicate one of the angles and the corresponding angle of the diagram with dark grey.}
\label{figure-2-3.pdf}
\end{figure}

\end{example}

\subsection{Vortex broken lines, 
their tangents and LLS sequences} 

Let us give a formal definition of a broken line bending counterclockwise with respect to some point.

\begin{definition}
We say that a broken line $A_1\ldots A_n$ 
is a {\it vortex broken line about $O$} if 
$$
\det (OA_i,OA_{i+1})>0
\qquad
\hbox{for $i=1,\ldots,n-1$.}
$$
\end{definition}

It turns out that the notion of LLS sequences is naturally extended to vortex broken lines (for further details see Chapter~15 of~\cite{karpenkov2022geometry}).

\begin{definition}\label{lls-brokenline-def}
For some integer vortex broken line $A_0\dots A_m$ we define the \textit{LLS sequence} of the broken line as the sequence $a_0,\dots,a_{2n}$ where
$$
\begin{array}{rcl}
a_{2k}&=&\det(OA_k,OA_{k+1})=\il(A_kA_{k+1}) \quad \hbox{for $k=0,1,\ldots, n$};
\\
a_{2k-1}&=&\displaystyle\frac{\det (A_{k+1}A_{k},A_{k+1}A_{k+2})}{a_{2k-1}a_{2k+1}}
\quad \hbox{for $k=1,2,\ldots,n$}.
\end{array}
$$
We also write $\lls(A_0\dots A_m)=(a_0,\dots,a_{2n})$.
\end{definition}

Let us finally recall the following important expression for the vertex coordinates of vortex broken lines (in fact the original statement is for a broader selection of broken lines).

\begin{theorem}{\bf(\cite[Theorem~15.10]{karpenkov2022geometry})}
\label{Geometry-of-broken-lines}
 Consider a vortex broken line $A_0\ldots A_{n}$ about the origin
with LLS sequence
$$
(a_0,a_1,\ldots,a_{2n-2}).
$$
Let also $A_0=(1,0)$ and $A_1=(1,a_0)$. Then
$$
A_i=\big(K(a_1,\ldots,a_{2i-2}),K(a_0,a_1,\ldots,a_{2i-2})\big) \quad \hbox{for $i=1,\ldots, n$}.
$$
 \qed
\end{theorem}

\begin{remark}
The LLS sequence is the complete invariant of integer congruence types of integer winding broken lines. 
Namely, two integer winding broken lines are integer congruent if and only if they have the same LLS sequence. In addition, an arbitrary sequence of odd length of integers with positive odd elements is the LLS sequence for some integer winding broken line. 
\end{remark}

\begin{remark}
In accordance with Definition~\ref{angle-sequence}, we have
$$
\angle (a_0,a_1,\ldots,a_{2n-2})
\cong\angle A_{0}OA_{n}.  
$$
\end{remark}

We conclude this section with the following general proposition.

\begin{proposition}
Any sail diagram is a vortex broken line about the origin.
In addition, all the edges of the sail diagram at at the unit distance to the origin.
\end{proposition}

\begin{proof}
Note that the sail for any integer angle is a vortex broken line with edges at the unit distance to the origin.
Since the sail diagram is the union of several sails, it also satisfies the conditions of the proposition.
\end{proof}

\subsection{Winding number of sail diagrams and its properties}

Denote by $\mu(\varphi)$ the Euclidean angular measure of the the angle $\varphi$. 

\begin{definition}
The {\it winding number} of a vortex broken line $A_0\ldots A_n$ about a point $O$ is the following number
$$
\omega(A_0\ldots A_n, O)=
\frac{1}{2\pi}\sum\limits_{i=1}^n \mu(\angle A_{i-1}OA_i).
$$
\end{definition}

In general, the winding number is not an invariant of lattice geometry except for the case when it is half an integer.
(The last correspond to the projectivisation of the Gauss map to the projective circle, and therefore it is a topological invariant). That is precisely the case of the sail diagrams of convex polygons.

\begin{proposition}\label{proposition-n/2-1}
The winding number of the sail diagram of the convex $n$-gon
equals $n/2-1$.
\end{proposition}

\begin{remark}
In particular, the winding number of the sail diagram of a convex polygon is its invariant in integer geometry.   
\end{remark}

\begin{proof}
It is clear that the sum of angles for the sail diagram equals the sum of angles of the $n$-gon. Namely, it is equivalent to $(n-2)\pi$ and hence the winding number equals $n/2-1$.
\end{proof}

\begin{proposition}\label{winding-vortex}
Let $A_0\ldots A_n$ be a sail diagram for an $n$-gon with
an angle-curvature sequence 
$$
\mathcal{S}=(\alpha_1,\aee_1,\ldots, \alpha_{n-1},\aee_{n-1}, \alpha_n, \aee_n).
$$ 
Then, the the number of sign changes for the sequence $(\lls(\mathcal S_1^j))_{j=1}^{n}$ equals
twice the winding number of $A_0\ldots A_n$ about the origin
minus 1. 
$($Here $\mathcal S_1^j$ are as in Definition~\ref{mathcal-S-n}.$)$
\end{proposition}

\begin{proof}
Let us first consider a sail diagram whose first two vertices are $A_0=(1,0)$ and $A_1=(1,a_0)$. 

By Theorem~\ref{Geometry-of-broken-lines}, the sequence $(\lls(\mathcal S_1^j))_{j=1}^{n}$
coincides with the sequence of $y$ coordinates of the vertices of the broken line connecting all the edge vertices of the sail diagram.
Since the original polygon is convex, this 
broken line makes precisely $n-2$ half-twists around the origin.
Again, due to convexity of the original polygon, we always add angles that are less than $\pi$ and hence the sequence of $y$ coordinates changes sign precisely $n-3$ times. 
Now the proof is concluded by Proposition~\ref{proposition-n/2-1}.

\vspace{2mm}

Let us now study the case of general sail diagrams. First of all, note that
such sail diagrams are integer congruent to some of the diagrams considered above. 
Secondly, both the winding number and the LLS sequence (and hence the number of sign changes) are invariants of integer geometry. 
This concludes the proof in the general case.
\end{proof}

\begin{remark}{\bf Cusps of sail diagrams and polygons.}
We say that the sail diagram has a cusp in edge vertex $B_i$ if
the chord curvature $\aee_i<0$.
Note that any broken line with only positive numbers in the LLS sequence winds less than $\pi$ around the origin. 
Since any $n$-gon does $n-2$ half twists around the origin starting from any vertex (this means that we forget one of the cusps),
the sail diagram has at least $n-1$ cusp.
In case of the number of cusps of the sail polyhedron is smaller than $n-1$, 
the corresponding broken line is still locally convex, however the sum of its angles is smaller than $\pi(n-3)$, and therefore it has self-intersections.
\end{remark}

\subsection{A few words about special integer points on sail diagrams}

Let formulate the following variation of the general classic matrix identity for continuants. 

\begin{proposition}\label{MatrixForm}
Consider an arbitrary sequence of numbers $(b_0,\ldots, b_k)$
and set $b_{-1}=0$.
Then the following identity holds:
$$
\begin{pmatrix}
    K(b_1,\ldots, b_{k}) \\
    K(b_0,\ldots, b_{k})
\end{pmatrix}
=
M_k
\begin{pmatrix}
1 \\
0
\end{pmatrix},
\qquad \hbox{where} \quad
M_k=\prod\limits_{i=0}^{k+1}
\begin{pmatrix}
    0&1 \\
    1&b_{k-i}
\end{pmatrix}
.
\qed
$$
\end{proposition}

Let us fix some notation.
Denote by $B_i(p_i,q_i)$ the edge vertices of the sail diagram.
Denote also the next and the previous integer points with respect to $B_i$ on the sail diagram by $B^+_i$ and by $B^-_i$ respectively.
Let 
$$
\lls(\mathcal S ^i_1)=(b_0,\ldots, b_{2s(i)}),
$$
and let $b_{2s(i)+1}$ be the next element of the LLS seqience for the sail diagram.
(Note that $s(i)$ is the total number of edges in the sails of the first $i$ integer angles.)

\vspace{2mm}

From the general theory of LLS sequences for broken lines we have:
$$
\begin{array}{l}
B_i= \big(K(b_1,\ldots, b_{2s(i)-1},b_{2s(i)}),
         K(b_0,\ldots, b_{2s(i)-1},b_{2s(i)})\big);\\
B_i^-= \big(K(b_1,\ldots, b_{2s(i)-1},b_{2s(i)}{-}1),K(b_0,\ldots, b_{2s(i)-1},b_{2s(i)}{-}1)\big);\\
B_i^+=
\big(K(b_1,\ldots, b_{2s(i)-1},b_{2s(i)},b_{2s(i)+1},1),K(b_0,\ldots, b_{2s(i)-1},b_{2s(i)},b_{2s(i)+1},1)\big).
\end{array}
$$
Set also 
$$
\hat B_i=
\big(K(b_1,\ldots, b_{2s(i)-1},b_{2s(i)},0,1),K(b_0,\ldots, b_{2i-1},b_{2s(i)},0,1)\big).
$$
(See Figure~\ref{figure-7.pdf}.)
\begin{figure}
    \centering
    \includegraphics[height=5cm]{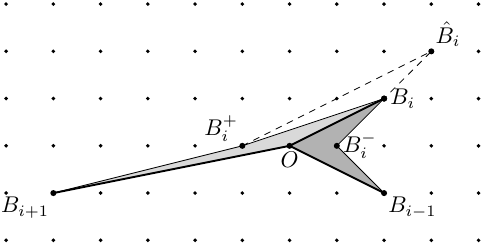}
    \caption{
Definition of points $B_i^-$, $B_i^+$, and $\hat B_i$.}
\label{figure-7.pdf}
\end{figure}

\begin{proposition}\label{properties-aeeb}
The following statements hold:

$($i$)$ $b_{2s(i)+1}=\det(B_iB_i^-,B_iB_i^+)$;

$($ii$)$ $b_{2s(i)+1}=\aee_i$;

$($iii$)$ $\hat B_i=B_i+ B_i^-B_i$;

$($iv$)$ $b_{2s(i)+1}=\det(B_iB_i^+,B_i\hat B_i)$;

$($v$)$ The points $B_i^+$ and $\hat B_i$ are in a line parallel to the line $OB_i$ at the unit integer distance to $OB_i$.

$($vi$)$ The second coordinate of $\hat B_{i}$ is $K(\lls(\mathcal{S}_1^{i})\concat (1))$.

$($vii$)$ The second coordinate of $B_{n-1}^+$  is in the interval between  $0$ and $K(\lls(\mathcal{S}_1^{n-1}))$.
It can take zero extreme value but not the other one.

\end{proposition}

\begin{proof}
{\it $($i$)$}
From general theory of continued fractions for broken lines (see e.g. in~\cite{Karpenkov2008}), the point $B_i^-$ is the last integer point of the last edge of the sail for $\alpha_i$; and 
the point $B_i^+$ is the last integer point of the last edge of the sail for $\alpha_{i+1}$ and hence, by the definition of the LLS sequence we have
$$
b_{2s(i)+1}
=\frac{\det(B_iB_i^-,B_iB_i^+)}
{\il(B_iB_i^-)\cdot \il(B_iB_i^+)}
=\frac{\det(B_iB_i^-,B_iB_i^+)}
{1\cdot 1}
=\det(B_iB_i^-,B_iB_i^+).
$$

\noindent {\it $($ii$)$}
Both  $b_{2i+1}$ and $\aee_i$
are invariants of integer geometry and hence without loss
of generality we restrict ourselves to the case (by changing the lattice coordinates):
$$
B_i^-=(b_{2s(i)+1}+1,-1),
\quad 
B_i=(1,0),
\quad 
B_i^+=(0,1),
\quad 
$$
On the one hand, we have
$$
\det(B_iB_i^-,B_iB_i^+)=b_{2s(i)+1}.
$$
On the other hand, let us
consider the edge $A_iA_{i+1}$ in the polygon for the given sail diagram 
corresponding to the edge vertex $B_i$ (see Figure~\ref{figure-5-6.pdf}). The edge $A_iA_{i+1}$ is parallel to the $x$-axes.
Without loss of generality, we consider 
$A_i=(0,0)$ and $A_{i-1}=(a,0)$
where $a=\il(A_{i-1}A_{i})$.
Then, the points $A_i'$ and $A_{i-1}'$
as in Definition~\ref{chord-curvature-ABCD}
will be the last and the first integer points on the sails for the angles 
$\angle A_{i-2}A_{i-1}A_{i}$
and 
$\angle A_{i-1}A_{i}A_{i+1}$
respectively.
Hence, they are
$$
\begin{array}{l}
A_{i-1}'=A_{i-1}-OB_i^-=(a,0)+(-b_{2s(i)+1}-1,1)
=(a-b_{2s(i)+1}-1,1);
\quad \hbox{and}\\
A_{i}'=A_{i}+OB_i^+=(1,0)+(0,1)
=(1,1).
\end{array}
$$
Then, by the definition of chord curvature we get
$$
\begin{aligned}
\aee_i
&=
\il(A_{i-1}A_i)-
\sgn \langle A_{i-1}A_i, A_{i-1}'A_i'\rangle
\il(A_{i-1}'A_i')-2\\
&=
a-(a-b_{2s(i)+1}-2)-2=b_{2s(i)+1}.    
\end{aligned}
$$

Therefore,
$$
\aee_i=\det(B_iB_i^-,B_iB_i^+)=b_{2s(i)+1}.
$$

\begin{figure}
    \centering
    \includegraphics[height=5.5cm]{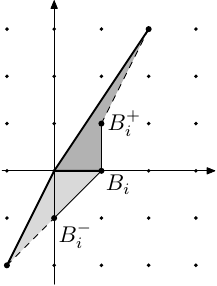}
    \hspace{3em}
    \includegraphics[height=5.5cm]{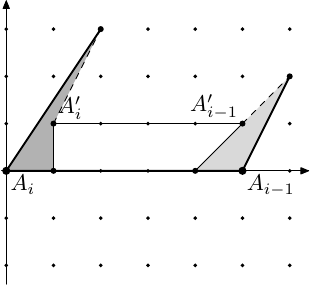}
    \caption{
Two consecutive angles (Left) in the sail diagram and the corresponding edge of the polygon (Right).}
\label{figure-5-6.pdf}
\end{figure}

\noindent {\it $($iii$)$}
Denote
$$
M=
\begin{pmatrix}
    0&1 \\
    1&b_{2s(i)}-1
\end{pmatrix}
\prod\limits_{j=1}^{2s(i)}
\begin{pmatrix}
    0&1 \\
    1&b_{2s(i)-j}
\end{pmatrix}
.
$$
Then, by Proposition~\ref{MatrixForm}, we have
$$
B_i^-=
M
\begin{pmatrix}
1 \\
0
\end{pmatrix},
\quad 
B_i=
\begin{pmatrix}
    0&1 \\
    1&1
\end{pmatrix}
\begin{pmatrix}
    0&1 \\
    1&0
\end{pmatrix}
M
\begin{pmatrix}
1 \\
0
\end{pmatrix},
\quad 
\hat B=
\begin{pmatrix}
    0&1 \\
    1&2
\end{pmatrix}
\begin{pmatrix}
    0&1 \\
    1&0
\end{pmatrix}
M
\begin{pmatrix}
1 \\
0
\end{pmatrix}.
$$
Since $M\in \GL(2,\Z)$, the statement is equivalent to the statement for the points
$$
\begin{pmatrix}
1 \\
0
\end{pmatrix},
\quad 
\begin{pmatrix}
    0&1 \\
    1&1
\end{pmatrix}
\begin{pmatrix}
    0&1 \\
    1&0
\end{pmatrix}
\begin{pmatrix}
1 \\
0
\end{pmatrix}
=
\begin{pmatrix}
1 \\
1
\end{pmatrix}
,
\quad 
\begin{pmatrix}
    0&1 \\
    1&2
\end{pmatrix}
\begin{pmatrix}
    0&1 \\
    1&0
\end{pmatrix}
\begin{pmatrix}
1 \\
0
\end{pmatrix}
=
\begin{pmatrix}
1 \\
2
\end{pmatrix},
$$
which obviously holds.

\vspace{2mm}

\noindent {\it $($iv$)$}
Since $B_iB_i^+=B_i^-B_i$, the values of the determinants coincide, so we have:
$$
b_{2s(i)+1}=\det(B_iB_i^-,B_iB_i^+)
=\det(B_iB_i^+,B_i\hat B_i).
$$

\noindent {\it $($v$)$}
The points $B_i^+$ and $B_i^-$ are at unit distance from the segment $OB_i$ since they are the first and the last points of the corresponding sails. They are in different half-planes with respect to the line containing $OB_i$.
By Item~({\it iii}), the point $\hat B_i$
is at unit distance to $OB_i$ in the different hyperplane to the point $B_i^-$.
Therefore, the points $B_i^+$ and $\hat B_i$ are in a line parallel to the line $OB_i$ at the unit integer distance to it.

\vspace{2mm}

\noindent {\it $($vi$)$} The second coordinate of $B_i^+$ it is equivalent to
$$
\begin{aligned}
K(b_0,\ldots, b_{2s(i)-1},b_{2s(i)},b_{2s(i)+1},0,1)
&=
K(b_0,\ldots, b_{2s(i)-1},b_{2s(i)},b_{2s(i)+1},1)\\
&=K(\lls(\mathcal{S}_1^{i})\concat (1)).
\end{aligned}
$$

\noindent {\it $($vii$)$}
Note that $B_i^+$ is the first integer point of sail of the angle $\angle B_iOB_0$.
Therefore, the second coordinate of $B_i^+$ has the same sign as $B_i$ (or $B_i^+$ can be zero), the absolute value of the second coordinate of $B_i^+$ 
is smaller than the absolute value of the second coordinate of $B_i$. This concludes the proof.

\end{proof}

\begin{corollary}\label{LLS-LLS-corollary}
The LLS sequence of the sail diagram of a convex integer polygon $P$
coincides with the the LLS sequence for the angle-curvature sequence of $P$.
\end{corollary}

\begin{proof}
By construction of the sail diagram, all the elements of the sails of angles coincide with the corresponding elements for the LLS sequences of angles in the angle-curvature.

\vspace{1mm}

By Proposition~\ref{properties-aeeb}(i) and (ii), we have
$$
\aee_i=\det(B_iB_i^-,B_iB_i^+).
$$
The last expression is precisely the expression for the corresponding angle in Definition~\ref{lls-brokenline-def}
after rescaling the corresponding vectors.

\vspace{1mm}

Therefore, the LLS sequence of a the sail diagram 
coincides with the the LLS sequence for the angle-curvature sequence.
\end{proof}

\begin{corollary}\label{aee-formula-corollary}
It holds
$$
 \aee_{n-1} = -\left\lfloor \frac{K(\lls(\mathcal S_1^{n-1})\concat (1))}{K(\lls(\mathcal S_1^{n-1}))} \right\rfloor.
$$
\end{corollary}

\begin{proof}
By Proposition~\ref{properties-aeeb}(vi), 
the numerator of the expression on the right hand side contains the second coordinate of $\hat B_{n-1}$.
By construction, the denominator contains the second coordinate of $B_{n-1}$.
Finally, by Proposition~\ref{properties-aeeb}(vii), the second coordinate of  $\hat B_{n-1}^+$ is
in the segment $[0,K(\lls(\mathcal{S}_1^{n-1}))$ excluding the right end.

\vspace{1mm}

By Proposition~\ref{properties-aeeb}$($v$)$, the points $B_{n-1}^+$ and $\hat B_{n-1}$ are in a line parallel to the line $OB_i$ on the unit integer length to it.
Therefore, computing the second coordinates we have
$$
\det(B_{n-1}B_{n-1}^+,B_{n-1}\hat B_{n-1})=
-\left\lfloor \frac{K(\lls(\mathcal S_1^{n-1})\concat (1))}{K(\lls(\mathcal S_1^{n-1}))} \right\rfloor.
$$
It remains to note that, by Proposition~\ref{properties-aeeb}$($iv$)$, it holds
$$
\det(B_{n-1}B_{n-1}^+,B_{n-1}\hat B_{n-1})=\aee_{n-1}.
$$
This concludes the proof.
\end{proof}

\section{Proofs of the main results}
\label{Proofs of the main results}

In this section we conclude the proofs of Theorem~\ref{IKEA-solution} and Theorem~\ref{IKEA-2}.

\subsection{Proof of Theorem~\ref{IKEA-solution}}\label{main-proof-Th1}

\noindent{\bf From a polygon to three conditions.}
Consider a convex $n$-gon with the angle-curvature
sequence $\mathcal S$.

Without loss of generality, we may assume that the first two vertices of the corresponding sail diagram are $A_0=(1,0)$ and $A_1=(1,a_0)$.
First of all, by Corollary~\ref{LLS-LLS-corollary}
the LLS sequence of the sail diagram 
coincides with the LLS sequence for the angle-curvature sequence.
Then, by Theorem~\ref{Geometry-of-broken-lines}, the sequence $(\lls(\mathcal S_1^j))_{j=1}^{n}$
coincides with the sequence of $y$ coordinates of the broken line connecting all the edge vertices of the sail diagram.
Since the last point of the sail diagram is $((-1)^n,0)$,
the first condition is immediate.

\vspace{2mm}

The second condition follows directly from 
Corollary~\ref{aee-formula-corollary}
(after shifting all the indices by 1).

\vspace{2mm}

The last condition follows directly from Proposition~\ref{winding-vortex}.

\vspace{2mm}

\noindent{\bf From three conditions to a polygon.}
Let us now construct a polygon with integer angles integer congruent to given ones: $\alpha_1, \ldots, \alpha_n$
and with corresponding chord curvatures $\aee_1,\ldots, \aee_{n}$ satisfying all the conditions of Theorem~\ref{IKEA-solution}.

\vspace{2mm}

Let us write the angle-curvature sequence
$$
\mathcal{S}=(\alpha_1,\aee_1,\ldots, \alpha_{n-1},\aee_{n-1}, \alpha_n)
$$ 
and the corresponding LLS-sequence for it.
Starting with vertices $(1,0)$ and $(1,a)$,
we construct the broken line with vertices whose coordinates are defined from this LLS-sequence by the formulae of Theorem~\ref{Geometry-of-broken-lines}. By the first and the third conditions, this broken line has the second endpoint at the line $y=0$ and its winding number is $n/2-1$.

\vspace{2mm}

So we have constructed a collection of angles which we shift (and take centrally symmetric for even angles) by integer vectors in order to get a convex $n$-gon with prescribed angle-curvature sequence. 
First, we find an $n$-gon whose angles are shifted by arbitrary real vectors (we omit the proof of this classical statement on Euclidean geometry here). Since all vectors are proportional to integer vectors, there exist a small parallel perturbation of the $n$-gon, which has rational coordinates. 
Finally, we arrive to an integer $n$-gon by scaling the rational polygon (multiplying by all denominators of rational coordinates of all vertices).

\vspace{2mm}

The obtained $n$-gon has the prescribed sail diagram.
It remains to note that the LLS-sequence of this sail diagram coincide with the LLS-sequence for the angle-curvature of the $n$-gon by Propositoin~\ref{winding-vortex}. This completes the proof.
\qed

\begin{remark}
The construction used in the proof 
suggests that there are infinitely many 
distinct integer types of polytopes with a given angle-curvature sequence. Note also that two $n$-gons with the same angle-curvature sequence are {\it parallel to each other} (after an integer-affine transformation of one of them), namely the sides of the first $n$-gon are parallel to the corresponding sides of the second $n$-gone.
The situation here is similar to the Euclidean case.
\end{remark}

\vspace{2mm}

\subsection{Proof of Theorem~\ref{IKEA-2}}

The expression for $x$ and $y$ follow directly from 
Corollary~\ref{aee-formula-corollary}.

\vspace{2mm}

The expression for $\beta$ is tautological. Indeed the sail diagram $B$ defining $\beta$ has the LLS sequence  $\lls(\mathcal{S}_1^n)$.
Then, let us show that $\beta \cong \angle\big(-U^t\big)$
where
$U=\lls(\mathcal{S}_1^n)$.

\vspace{1mm}

Consider the angle $-\beta$
(this is the angle symmetric to the angle $\beta$ about the $x$-axis with the first edge at the origin).
Then we take the broken line $B'$ symmetric about the $x$-axis to the broken line $B$.
By Definition~\ref{lls-brokenline-def}, 
the LLS-sequence for $B$ is precisely $-U$. Hence 
by Theorem~\ref{Geometry-of-broken-lines}
$$
-\beta\cong \angle (B')=\angle (-B)\cong\angle(-U).
$$
Finally, we have
$$
\beta\cong -\beta^t \cong \big(\angle (-U)\big)^t=\angle (-U^t)
$$
as it is stated in the theorem.
\qed

\section{A few words on the multidimensional case}
\label{A few words on the multidimensional case}

We conclude this paper with a small discussion of the open IKEA problem in the multidimensional case.
\begin{problem}{\bf (IKEA problem in higher dimensions.)}
Classify all angles in lattice convex polyhedra in dimensions greater than 2.
\end{problem}

{\it How to generalize the statements of Theorems~\ref{IKEA-solution} and~\ref{IKEA-2}?}
In order to approach this problem one might consider some multidimensional continued fractions.

\vspace{2mm}

There are at least two competing theories of multidimensional continued fractions. The first is based on geometry of lattices. 
Here again we have at least two different version of continued fractions  that
may be used here: Klein polyhedra~\cite{Klein1895,Klein1896,Arnold1980,karpenkov2022geometry}
and Minkovski-Voronoi polyhedra~\cite{Minkowski1896,Voronoi1896, Minkowski1967,KU2017}.
The second approach is based on the algorithmic theory of multidimensional continued fractions~\cite{Schweiger2000,Karpenkov2022-Hermite}.
(Some further problems on geometric continued fractions 
are collected in~\cite{Karpenkov2017-problems}.)
The problem is open in all the above settings.
The progress in it will contribute to the study of singularities of toric varieties, see, e.g.~\cite{Tsuchihashi1983}.

\vspace{2mm}

Finally, we would like to mention a recent progress in the study of lattice trigonometry for cones in higher dimensions in~\cite{BDK2023}.

\medskip

\printbibliography

\end{document}